\documentclass[12pt,onecolumn,draftcls,journal,letterpaper]{IEEEtran}
\usepackage{hyperref}
\pdfoutput=1

\usepackage{amssymb,amsmath,color}
\usepackage{graphicx}
\usepackage{epsfig}

\renewcommand{\natural}{{\mathbb{N}}}
\newcommand{\naturalzero}{\mathbb{N}_0}

\newcommand{\real}{{\mathbb{R}}}

\newcommand{\setdef}[2]{\{#1 \; | \; #2\}}

\newcommand{\NN}{\mathcal{N}}

\newcommand{\RR}{\mathcal{R}}

\newcommand{\until}[1]{\{1,\dots,#1\}}
\newcommand{\fromto}[2]{\{#1,\dots,#2\}}

\newcommand{\degree}{\textup{\text{d}}}

\newcommand{\Lu}{N}
\newcommand{\Lm}{M}

\newcommand{\eqmod}[1]{{\buildrel\rm mod \, #1\over=}}

\newtheorem{theorem}{Theorem}[section]
\newtheorem{proposition}[theorem]{Proposition}
\newtheorem{corollary}[theorem]{Corollary}

\newtheorem{lemma}[theorem]{Lemma}
\newtheorem{remark}[theorem]{Remark}

\newcommand\oprocendsymbol{\hbox{$\square$}}
\newcommand\oprocend{\relax\ifmmode\else\unskip\hfill\fi\oprocendsymbol}

\renewcommand{\theenumi}{(\roman{enumi})}

\graphicspath{{figs/}}

\begin{document}

\title{On the reachability and observability\\ of path and cycle graphs
  \thanks{An early short version of this work appeared as~\cite{GP-GN:10b}:
    differences between this early short version and the current article include
    the reachability analysis, a much improved comprehensive and thorough
    treatment, revised complete proofs for all statements.}  }

\author{Gianfranco Parlangeli \and Giuseppe Notarstefano\thanks{The research
    leading to these results has received funding from the European Community's
    Seventh Framework Programme (FP7/2007-2013) under grant agreement no. 224428
    (CHAT) and n. 231378 (CO3AUV) and from the national project ``Sviluppo di
    nuovi metodi e algoritmi per l'identificazione, la stima bayesiana e il
    controllo adattativo e distribuito''.}  \thanks{Gianfranco Parlangeli and
    Giuseppe Notarstefano are with the Department of Engineering, University of
    Lecce, Via per Monteroni, 73100 Lecce, Italy,
    \texttt{\{gianfranco.parlangeli, giuseppe.notarstefano\}@unile.it}}}

\maketitle

\begin{abstract}
  In this paper we investigate the reachability and observability properties of
  a network system, running a Laplacian based average consensus algorithm, when
  the communication graph is a path or a cycle. More in detail, we provide
  necessary and sufficient conditions, based on simple algebraic rules from
  number theory, to characterize all and only the nodes from which the network
  system is reachable (respectively observable). Interesting immediate
  corollaries of our results are: (i) a path graph is reachable (observable)
  from any single node if and only if the number of nodes of the graph is a
  power of two, $n=2^i, i\in \natural$, and (ii) a cycle is reachable
  (observable) from any pair of nodes if and only if $n$ is a prime number. For
  any set of control (observation) nodes, we provide a closed form expression
  for the (unreachable) unobservable eigenvalues and for the eigenvectors of the
  (unreachable) unobservable subsystem.
\end{abstract}

\section{Introduction}

Distributed computation in network control systems has received great attention
in the last years. One of the most studied problems is the \emph{consensus
  problem}. Given a network of processors, the task of reaching consensus
consists of computing a common desired value by performing local computation and
exchanging local information. A variety of distributed algorithms for diverse
system dynamics and consensus objectives has been proposed in the literature.

We are interested in two problems that may arise in a network running a
consensus algorithm when only a subset of nodes is controlled by an external
input or measured by an external processor. Namely, is it possible to reach all
the node configurations just controlling a limited number of nodes?
Respectively, is it possible to reconstruct the entire network state just
knowing the state of a limited number of nodes?

In this paper we will concentrate on a first-order network system running a
Laplacian based average consensus algorithm with fixed communication graph
topology of path or cycle type. Average consensus has been widely studied in the
last years. Several distributed feedback laws have been proposed. A survey on
these algorithms and their performance may be found e.g. in
\cite{ROS-JAF-RMM:07} and references therein. The dynamical system arising from
a consensus network with fixed topology is a linear time-invariant system. The
problem of understanding if all the node configurations can be reached by
controlling a given subset of nodes is a reachability problem. Respectively, the
problem of understanding if the entire network state may be reconstructed is an
observability problem.

We organize the relevant literature in two areas. First, the reachability
problem in a first-order network arises in network systems where all (or most
of) the nodes run a linear average consensus algorithm, while a subset of them
can be driven by an exogenous input. These networks are often called
leader-follower networks in the sense that the nodes controlled by an external
input are leaders that drive the followers to desired configurations. The
reachability (controllability) problem for a leader-follower network was
introduced in \cite{HGT:04} for a single control node.
Intensive simulations were provided showing that it is ``unlikely'' for a
Laplacian based consensus network to be completely controllable.
In \cite{AR-MJ-MM-ME:09} and \cite{SM-ME-AB:09}, necessary conditions for
controllability, based on suitable properties of the graph, have been
provided. The conditions rely on algebraic graph tools based on the notion of
equitable partitions of a graph. A more exhaustive analysis of the
controllability and other structural system properties for network systems on
the basis of graph structural properties can be found in \cite{MM-ME:10}.
In \cite{BL-TC-LW-GX:08} preliminary results are given for the controllability
of first-order network with switching communication topology applied to the
formation control problem.
In \cite{ZJ-ZW-HL-ZW:10} sufficient conditions for controllability in first and
second order multi-agent systems with delay are given. The conditions are based
on the eigenstructure of the network system and delay matrices.
In \cite{ZJ-HL-THL-QL:10} necessary conditions and sufficient conditions for the
controllability of tree graphs are given, based on the eigenvalues of suitable
subgraphs.
A first contribution to the controllability of multi-agent systems with
nonlinear dynamics is provided in \cite{WH-LP-AB:10}, where the controllability
of pairs of identical nonholonomic vehicles maintaining a constant distance is
studied.

Second, observability for a network system running an average consensus
algorithm has been studied for the first time in \cite{MJ-ME:07}. In that paper
the authors provide necessary conditions for observability. The conditions are
based on equitable partitions of a graph as in the reachability setting
investigated in \cite{AR-MJ-MM-ME:09} and \cite{SM-ME-AB:09}.
A recent reference on observability for network dynamic systems is
\cite{DZ-MM:08}. Here, the linear dynamical systems of the network are decoupled
and the coupling among the systems appears through the
output.
A parallel research line investigates a slightly different property called
\emph{structural observability} \cite{SS-CH:08}. Here, the objective is to
choose the nonzero entries of the consensus matrix (i.e. the state matrix of the
resulting network system) in order to obtain observability from a given set of
nodes. However, in many contexts the structure of the system matrix is given
(e.g. the Laplacian for average
consensus).

It is worth noting that, the observability property is an important property in
distributed estimation, \cite{VG:06, MF-GFT-RS:09}, and intrusion detection
problems \cite{FP-AB-FB:10, SS-CNH:11} for steady state analysis. In the
literature this property is often assumed or considered as non generic.
Finally, preliminary results on the controllability and observability of path
and cycle graphs were given in \cite{RL-MWS-JAG-NC:08}, where, using these
properties, a formation control strategy was proposed.

The contribution of the paper is twofold. First, we provide \emph{necessary and
  sufficient} conditions based on simple algebraic relations from number theory
that completely characterize the reachability (observability) of path and cycle
graphs. More in detail, on the basis of the node labels and the total number of
nodes in the graph we are able to: (i) identify all and only the reachable
(observable) nodes of the graph, (ii) say if the graph is reachable (observable)
from a given set of nodes and (iii) construct a set of control (observation)
nodes from which the graph is reachable (observable).  Interesting immediate
corollaries of our results are: (i) a path graph is reachable (observable) from
any single node if and only if the number of nodes of the graph is a power of
two, $n=2^i, i\in \natural$, (ii) a path graph is reachable (observable) from a
single node $i$ if and only if there is no odd prime factor $p$ of $n$ such that
$(n-i)= (i-1) + \alpha p$ for some integer $\alpha$; (iii) a cycle graph is
reachable (observable) from any pair of nodes if and only if $n$ is a prime
number, and (iv) a cycle graph is reachable (observable) from two nodes, say
$i_1$ and $i_2$, if and only if $i_2-i_1$ and $n+i_1-i_2$ are coprime. Thus,
e.g., any cycle is observable from two adjacent nodes.

Second, we provide a closed form expression for the unreachable (unobservable)
eigenvalues and eigenvectors, and characterize the orthogonal complement to the
reachable subspace (respectively the unobservable subspace) for any unreachable
(unobservable) set of nodes. This result is based on the complete
characterization of the spectrum of suitable submatrices of the path and cycle
Laplacians. As a consequence of these linear algebra results, we also provide a
closed form for all the Laplacian eigevalues of a path graph. At the best of our
knowledge both the characterization of the Laplacian eigenvalues and the
mathematical tools used to characterize them are new.

The paper is organized as follows. In Section~\ref{sec:prelims} we introduce
some preliminary definitions and properties of undirected graphs, describe the
network model used in the paper and set up the reachability and observability
problems.
In Section~\ref{sec:cycle-path_spectrum} we provide a complete characterization
of the eigenvalues and eigenvectors of the Laplacian of a path graph and other
matrices used to study the path and cycle reachability (observability). Finally,
in Section~\ref{sec:main_results} we provide a complete characterization of the
path and cycle graph reachability (observability), and provide some useful
example explaining the main results.

\paragraph*{Notation}
We let $\natural$, $\naturalzero$, and $\real_{\geq0}$ denote the natural
numbers, the non-negative integer numbers, and the non-negative real numbers,
respectively. We denote $0_{d}$, $d\in\natural$, the vector of dimension $d$
with zero components and $0_{d_1\times d_2}$, $d_1, d_2\in \natural$, the matrix
with $d_1$ rows and $d_2$ columns with zero entries. For $i\in\natural$ we let
$e_i$ be the $i$-th element of the canonical basis, e.g. $e_1 = [1~ 0~ \ldots~
0]^T$. For a matrix $A\in\real^{d_1\times d_2}$ we denote $[A]_{ij}$ the
$(i,j)th$ element and $[A]_{i}$ the $i$th column of $A$. For a vector
$v\in\real^d$ we denote $(v)_i$ the $i$th component of $v$ so that $v = [(v)_1
\ldots (v)_d]^T$. Also, we denote $\Pi \in \real^{d\times d}$ the permutation
matrix reversing all the components of $v$ so that $\Pi v = [(v)_d \ldots
(v)_1]^T$ (the $j$-th column of $\Pi$ is $[\Pi]_j= e_{n-j+1}$).
Adopting the usual terminology of number theory, we will say that $k$
\textsl{divides} a nonzero integer $m$ (written $k|m$) if there is an integer
$q$ with the property that $m=kq$. When this relation holds, $k$ is said a
\textsl{factor} or \textsl{divisor} of $m$.  If two integers $b$ and $c$ satisfy
for a given $m$ the relation $m|(b-c)$ then we say that $b$ is congruent to $c$
modulo $m$ (written $b = c$ mod($m$) or equivalently $b \eqmod{m} c$).  The
\textsl{greatest common divisor} of two positive integers $a$ and $b$ is the
largest divisor common to $a$ and $b$, and we will denote it $GCD(a,b)$.  The
greatest common divisor can also be defined for three or more positive integers
as the largest divisor shared by all of them. Two or more positive integers that
have greatest common divisor $1$ are said \textsl{relatively prime} or
\textsl{coprime}.
A \textsl{prime number} is a positive integer that has no positive integer
divisors other than $1$ and itself. Every natural number $n$ admits a
\textsl{prime factorization} (Fundamental Theorem of Arithmetic), i.e. we can
factorize $n$ as $n = 2^{n_0} \prod_{\alpha = 1}^l p_\alpha$, where
$n_0\in\natural$ and each $p_\alpha$ is an odd prime number. Notice that in our
factorization we allow two or more factors $p_{\alpha}$ to be equal.


\section{Preliminaries and problem set-up}
\label{sec:prelims}
In this section we present some preliminary terminology on graph theory,
introduce the network model, set up the reachability and observability problems,
and provide some standard results on reachability (observability) of linear
systems that will be useful to prove the main results of the paper.

\subsection{Preliminaries on graph theory}
Let $G = (I, E)$ be a static undirected graph with set of nodes $I=\until{n}$
and set of edges $E\subset I\times I$. We denote $\NN_i$ the set of neighbors of
agent $i$, that is, $\NN_i = \setdef{j\in I}{(i,j)\in E}$, and $\degree_i =
\sum_{j\in \NN_i} 1$ the degree of node $i$. The maximum degree of the graph is
defined as $\Delta = \max_{i\in I} \degree_i$. The degree matrix $D$ of the
graph $G$ is the diagonal matrix defined as $[D]_{ii} = \degree_i$.
The adjacency matrix $A \in \real^{n\times n}$ associated to the graph $G$ is
defined as
\[
\left[ A \right]_{ij} =
\begin{cases}
  1 & \text{if}~ (i,j)\in E\\
  0 & \text{otherwise}.
\end{cases}
\]
The Laplacian $L$ of $G$ is defined as $L = D - A$. The Laplacian is a symmetric
positive semidefinite matrix with $k$ eigenvalues in $0$, where $k$ is the
number of connected components of $G$. If the graph is connected the eigenvector
associated to the eigenvalue $0$ is the vector $\mathbf{1} = [1~ \ldots~ 1]^T$.

We introduce the two undirected graphs that will be of interest in the rest of
the paper, namely the path and cycle graphs.
A \emph{path graph} is a graph in which there are only nodes of degree two
except for two nodes of degree one.
The nodes of degree one are called external nodes, while the other are called
internal nodes.
From now on, without loss of generality, we will label the external nodes with
$1$ and $n$, and the internal nodes so that the edge set is $E = \{(i,i+1) \; |
\; i\in\fromto{1}{n-1}\}$.
A \emph{cycle graph} is a graph in which all the nodes have degree two. From now
on, without loss of generality, we will label the nodes so that the edge set is
$E = \{(i,i\,\text{mod}(n) + 1) \; | \; i\in\fromto{1}{n}\}$.

\subsection{Network of agents running average consensus}
\label{sec:network_model}
We consider a collection of agents labeled by a set of identifiers $I =
\until{n}$, where $n\in\natural$ is the number of agents. We assume that the
agents communicate according to a \emph{time-invariant undirected} communication
graph $G = (I, E)$, where $E = \setdef{(i,j)\in I\times I}{i~ \text{and}~ j~
  \text{communicate}}$. That is, we assume that the communication between any
two agents is bi-directional.
The agents run a consensus algorithm based on a Laplacian control law (see
e.g. \cite{ROS-JAF-RMM:07} for a survey). The dynamics of the agents evolve in
continuous time ($t\in\real_{\geq0}$) and are given by
\begin{equation*}
  \begin{split}
    \dot{x}_i(t) = - \sum_{j\in \NN_i} (x_i(t) - x_j(t)), ~ i \in \until{n}.
  \end{split}
\end{equation*}
Using a compact notation the dynamics may be written as
\begin{equation*}
  \dot{x}(t) = -L x(t), ~t\in\real_{\geq0},
\end{equation*}
where $x = [(x)_1 \ldots (x)_n]^T = [x_1 \ldots x_n]^T$ is the vector of the
agents' states and $L$ is the graph Laplacian.

\begin{remark}[Descrete time system]
  In discrete time, we can consider the following dynamics
  \begin{equation*}
    \begin{split}
      x_i(t+1) = x_i(t) - \epsilon \sum_{j\in \NN_i} (x_i(t) - x_j(t)), ~ i \in
      \until{n},
    \end{split}
  \end{equation*}
  where $\epsilon \in \real$ is a given parameter. A compact expression for the
  dynamics is
  \begin{equation*}
    x(t+1) = (I - \epsilon L) x(t), ~t\in\naturalzero.
  \end{equation*}
  For $\epsilon \in (0, 1 / \Delta)$ ($\Delta$ is the maximum degree of the
  graph), $P$=$(I - \epsilon L)$ is a nonnegative, doubly stochastic, stable
  matrix.
	
  It can be easily shown that the continuous and discrete time systems have the
  same reachability and observability properties (namely the same unreachable
  and unobservable eigenvalues and eigenvectors). Therefore, the results shown
  in the paper also hold in this discrete time set-up. \oprocend
\end{remark}

\subsection{Network reachability and observability}
\label{sec:probl_statement}
In this section we describe the mathematical framework that we will use to study
the reachability and observability of a network system.
We start by describing the scenarios that motivate our work.
As regards the reachability problem, we imagine that in a network of agents
running average consensus as in Section~\ref{sec:network_model}, a subset of
nodes can be controlled by an external input. In the literature these nodes are
often called \emph{leader nodes} or \emph{pinned nodes}. The idea is that they
are special nodes with higher computation capabilities so that more
sophisticated control laws can be designed.
Let $I_\ell = \{i_1, \ldots, i_m\}\subset\until{n}$ be the set of control nodes,
the dynamical system modeling this scenario is
\begin{equation*}
  \dot{x}(t) = -L x(t) + B u(t),
  \label{eq:lead_follow_net}
\end{equation*}
where $u(t) = [u_{i_1}(t), \ldots, u_{i_m}(t)]^T$ is the input and $B = [e_{i_1}
\;| \; \ldots \;|\; e_{i_m}]$.

As regards the observability problem, we imagine that an external processor
(\emph{not} running the consensus algorithm) collects information from some
nodes in the network. We call these nodes \emph{observation nodes}. In
particular, we assume that the external processor may read the state of each
observation node.
Equivalently, we can think of one or more observation nodes, running the
consensus algorithm, that have to reconstruct the state of the network by
processing only their own
state.
We can model these two scenarios with the following mathematical
framework.
For each observation node $i \in I$, we have the following output
\begin{equation*}
  y_i(t) =  x_i(t).
\end{equation*}
Given the set of observation nodes $I_o=\{i_1, \ldots, i_m\}\subset\until{n}$,
the output is $y(t) =
\begin{bmatrix}
  x_{i_{1}}(t) & x_{i_{2}}(t)& \ldots & x_{i_{m}}(t)
\end{bmatrix}^T$.
Therefore, the system dynamics is given by
\begin{equation*}
  \begin{split}
    \dot{x}(t) &= -L x(t),\\
    y(t) &= C x(t),
    \label{eq:oberv_net}
  \end{split}
\end{equation*}
where the output matrix is $C = [e_{i_1} \;| \; \ldots \;|\; e_{i_m}]^T$.

It is a well known result in linear systems theory that the reachability
properties of the pair $(L, B)$ are the same as the observability properties of
the pair $(L, C) = (L^T, B^T) = (L, B^T)$.

\begin{remark}[Duality and regulator design]
  Due to the symmetry of the state matrix $L$, the network is reachable from a
  given subset of nodes if and only if it is observable from it. This important
  property allows, for example, to design a regulator at a subset of
  (control/observation) nodes that estimates the entire network state and
  controls it to a desired configuration. This has, for example, certainly an
  impact on security issues. \oprocend
\end{remark}

\begin{remark}[Equivalence with other problem set-ups]
  Straightforward results from linear system theory can be used to prove that
  the controllability problem studied, e.g., in \cite{AR-MJ-MM-ME:09} and
  \cite{SM-ME-AB:09} and the dual observability problem studied in
  \cite{MJ-ME:07} can be equivalently formulated in our set up. \oprocend
\end{remark}

\subsection{Standard results on reachability and observability of linear
  systems}
The reachability problem consists of looking for those states that can be
reached in finite time from the origin. Respectively, the observability problem
consists of looking for nonzero values of $x(0)$ that produce an identically
zero output $y(t)$. Using known results in linear system theory the two problems
are equivalent to studying the rank of the reachability matrix,
\[
\mathcal{R}_{n} = \left[ B \; | \; L B \; | \; \ldots \; | \; L^{n-1} B \right],
\]
respectively of the observability matrix
\[
\mathcal{O}_{n} =
\begin{bmatrix}
  C \\
  C L \\
  \vdots \\
  C L^{n-1}
\end{bmatrix}.
\]
The image $X_r$ of $\RR_n$ is the \emph{reachable subspace}, i.e. the set of
states that are reachable from the origin. The kernel $X_{no}$ of
$\mathcal{O}_n$ is the \emph{unobservable subspace}, i.e. the set of initial
states that produce an identically zero output.

Here, we recall an interesting result on the reachability (observability) of
time-invariant linear systems known as Popov-Belevich-Hautus (PBH) lemma.
\begin{lemma}[PBH lemma]
  Let $A \in \real^{n\times n}$, $B \in \real^{n\times m}$ and
  $C\in\real^{m\times n}$, $n,m \in \natural$, be the state, input and output
  matrices of a linear time-invariant system. The pair $(A, B)$ is reachable if
  and only if
  \begin{equation*}
    \text{rank}
    \begin{bmatrix}
      A - \lambda I \; | \; B
    \end{bmatrix}
    = n,
  \end{equation*}
  for all $\lambda \in \mathbb{C}$. Respectively the pair $(A, C)$ is
  observable, if and only if
  \begin{equation*}
    \text{rank}
    \begin{bmatrix}
      C\\
      A - \lambda I\\
    \end{bmatrix}
    = n,
  \end{equation*}
  for all $\lambda \in \mathbb{C}$.\oprocend
\end{lemma}

Combining the PBH lemma with the fact that the state matrix is symmetric (and
therefore diagonalizable) the following corollary may be proven.

\begin{corollary}[PBH lemma for symmetric matrices]
  Let $A\in \real^{n\times n}$, $B \in \real^{n\times m}$ and
  $C\in\real^{m\times n}$, $n,m \in \natural$, be the state, input and output
  matrices of a linear time-invariant system, where $A$ is symmetric. Then, the
  orthogonal complement to the reachable subspace $X_r$, associated to the pair
  $(A,B)$, is spanned by vectors $v_l$ satisfying
  \begin{equation}
    \begin{split}
      B^T v_l &= 0_m\\
      L v_l &= \lambda v_l.
    \end{split}
    \label{eq:PBH_eigvec_reach}
  \end{equation}
  Respectively, the unobservable subspace $X_{no}$ associated to the pair $(L,
  C)$ is spanned by vectors $v_l$ satisfying
  \begin{equation}
    \begin{split}
      C v_l &= 0_m\\
      L v_l &= \lambda v_l,
    \end{split}
    \label{eq:PBH_eigvec_obs}
  \end{equation}
  for $\lambda\in\real$. \oprocend
  \label{cor:PBH_eigvec}
\end{corollary}

In the rest of the paper we will denote the eigenvalues and eigenvectors for
which \eqref{eq:PBH_eigvec_reach} is satisfied, \emph{unreachable eigenvalues
  and eigenvectors}. Respectively, we will denote \emph{unobservable eigenvalues
  and eigenvectors} the ones for which \eqref{eq:PBH_eigvec_obs} is satisfied.

\section{Spectral properties of the Laplacian of a path and related submatrices}
\label{sec:cycle-path_spectrum}
In this section we provide a closed form expression for the eigenvalues and
eigenvectors of suitable submatrices of the Laplacian of path and cycle
graphs. This characterization will play a key role in the characterization of
the reachability (observability) properties of path and cycle graphs. As a
self-contained result of this section, we provide a closed form expression for
the Laplacian spectrum of a path graph.

We start motivating the analysis in this section.
Recall that, without loss of generality, we assume that nodes of the path are
labeled so that the undirected edges are ($i$, $i+1$) for $i\in\until{n-1}$.
Let $L_n$ denote the Laplacian of a path graph of length $n$ and $B$ ($C$) the
input (output) matrix associated to the set of control (observation) nodes $I_o
= \{i_1,\ldots,i_m\}$. Using the PBH Lemma in the version of
Corollary~\ref{cor:PBH_eigvec}, unreachability (unobservability) for the path
graph from $I_o$ is equivalent to the existence of a nonzero solution of the
linear (algebraic) system $L_n v = \lambda v$, where $v$ satisfies $B^T\, v = 0$
($C\, v = 0$), i.e. $(v)_j = 0$ for each $j\in I_o$. Exploiting the structure of
the linear system we can write

\[ \!\!\!\!\!\!\!\!\!\!\!\!\!\!\!\!\!\!\!\!  \!\!\!\!\!\!\!\!\!\!\!\!\!\!
\!\!\!\!\!\!\ \!\!\!\!\!\!\!\!\!\!\!  \!\!\!\!\!\!\!\!\!\!\!\!\!\!\!\!\!\!\!\!
\!\!\!\!\!\!  \!\!\!\!\!\!\!\!\!\!\!\!\!\!\!\!
\begin{bmatrix}
  1      &      -1 &\dots        &     0    &   \dots & 0 \\
  -1      &       2 &     \       &     \vdots    &  \   &     0\\
  \vdots &         & \ddots      &      -1      & \   &    0\\
  0      &         &    -1       &       2  &      \ &    0\\
  0      &         &    0       &      -1  &    \ &     -1\\
  0      &         &            &   \vdots    &   \dots &       1\\
\end{bmatrix}\!\!\!
\begin{bmatrix}
  (v)_1  \\
  \vdots\\
  (v)_{i_1-1}\\
  0\\
  (v)_{i_1+1}\\
  \vdots\\
  (v)_n  \\
\end{bmatrix}
\!\!\!=\! \lambda\!\!
\begin{bmatrix}
  (v)_1  \\
  \vdots\\
  (v)_{i_1-1}\\
  0\\
  (v)_{i_1+1}\\
  \vdots\\
  (v)_n  \\
\end{bmatrix}\!\!\!\!  \!\!\!\!\!\! \!\!\!\!\!\!\!\!\!\!\!\!\!
\!\!\!\!\!\!\!\!\!\!\!\!\!\!\!\!\!  \!\!\!\!\!\!\!\!\!\!\!\!\!\!\!\!\!\!\!\!
\!\!\!\!\!\!\!\!\!\!\!\!\!\!\!\!  \!\!\!\!\!\!\!\!\!\!\!\vrule
\]

\small
$$\!\!\!\!\!\!\!\!\!\!\!\!\!\!\!\!
\uparrow i_1\textrm{-th column }$$ \normalsize
where the vertical line on the $i_1$-th column means that it is multiplied by
$(v)_{i_1}=0$. The same holds for each $j$-th column, $j\in I_o$.
Now, define the matrices $\Lu_\nu\in \real^{\nu\times\nu}$ and $\Lm_\mu\in
\real^{\mu\times\mu}$ as

\[
\Lu_\nu = 
\begin{bmatrix}
  1 & -1 &   & 0 \\
  -1&  2 &   &   \\
  0 &   & \ddots  &  -1\\
  0 &   & -1 & 2 \\
\end{bmatrix} \enspace \text{and} \enspace \Lm_\mu = 
\begin{bmatrix}
  2 & -1 &   & 0 \\
  -1&  2 &   &   \\
  0 &   & \ddots  &  -1\\
  0 &   & -1 & 2 \\
\end{bmatrix},
\]
where the subindex refers to their dimensions. The Laplacian $L_n$ can be
compactly written as

\[
\hspace{-3cm}
\begin{split}
  L_n &=\begin{bmatrix} ~ & \phantom{-}\vdots &\dots & \phantom{-}0 & \dots &
    0 \\
    \Lu_{i_1-1} & -1 &\dots & \phantom{-}0 & \dots &
    0 \\
    \ldots & \phantom{-}2 & \ & \phantom{-}\vdots & \ &
    0\\
    \vspace{-0.5cm} & -1 & ~ & ~ & \
    &    ~\\
    \vspace{-0.4cm}~ & ~ & \Lm_{i_2-i_1-1} & ~ & \
    &    ~\\
    \vdots & ~ & ~ & -1 & \
    &    ~\\
    0 & \phantom{-}0 & & \phantom{-}2 & \ &
    0\\
    0 & \phantom{-}\vdots& 0 & -1 & \ddots\
    &      \\
    0 & & & \phantom{-}\vdots & &
    \Pi\Lu_{n-i_m}\Pi \\
  \end{bmatrix}, \hspace{-6.39cm}\vrule
  \hspace{2.75cm}\vrule\\
  &\tiny \hspace{1.2cm}\textrm{$i_1$-th column} \uparrow
  \hspace{0.4cm}\textrm{$(i_2-i_1)$-th column} \uparrow \dots \normalsize
\end{split}
\]
where $\Pi = \Pi^T = \Pi^{-1}$ is the (symmetric) permutation matrix reversing
all the components of a vector.

\begin{remark}[Partition of the Laplacian of a cycle]
  Applying the same procedure to the Laplacian of a cycle, under the agreement
  of labeling the nodes so that $i_1 = 1$, we get a partition of the Laplacian
  where the submatrices are all matrices of type $\Lm_\mu$, $\mu\in
  \natural$.\oprocend
\end{remark}

We are now ready to investigate the spectral properties of these matrices. We
begin with a useful lemma.

\begin{lemma}
  \label{lem:nonzero_comp}
  The eigenvectors of $\Lu_\mu$, $\Lm_\mu$ and $L_\mu$, $\mu\in \natural$, have
  nonzero first and last components. 
\end{lemma}

\begin{proof}
  We prove the statement by proving that if an eigenvector has zero first
  component, then all the other components have to be zero. The same line of
  proof can be followed to prove that if the last component is zero, then the
  same must hold for the previous $\mu-1$.
  
  Let $\lambda$ be an eigenvalue of $\Lu_\mu$ (respectively $\Lm_\mu$ and
  $L_\mu$) and $v = [v_1 \ldots v_\mu]$ the corresponding eigenvector. The
  following conditions must hold
  \[
  \begin{split}
    a v_1 - v_2 &= \lambda v_1\\
    -v_{i-1} + b v_{i} -v_{i+1} &= \lambda v_i, \qquad i\in\fromto{2}{\mu-1}\\
    -v_{\mu-1} + b v_{\mu} &= \lambda v_\mu,
  \end{split}
  \]
  for suitable (positive) $a$ and $b$.  We proceed by induction. We use the
  inductive assumption that $v_{j-1}=v_{j}=0$ and prove that $v_{j+1}=0$. The
  statement is obviously true for $j=1$. Indeed, from the first equation it
  follows easily that if $v_1=0$ then $v_2=0$ (we have considered a fake $v_0=0$
  in the first equation). Then, plugging the inductive assumption in the second
  equation, we get $v_{j+1}=0$ for $j\in\fromto{2}{\mu-1}$, while the last
  equation gives the result for $j=\mu$.
\end{proof}

\begin{remark}[A path is reachable (observable) from an external node]
  Combining the previous lemma with Corollary~\ref{cor:PBH_eigvec}, it follows
  easily that a path graph is reachable (observable) from each of the external
  nodes as shown, e.g., in \cite{AR-MJ-MM-ME:09} (\cite{GP-GN:10a}). \oprocend
  \label{rmk:reach_observ_external_nodes}
\end{remark}

\begin{proposition}[Eigenstructure of $\Lu_\nu$ and $\Lm_\mu$]
  \label{prop:eigs_Lu_Lm}
  For any two matrices $\Lu_\nu\in \real^{\nu\times\nu}$ and $\Lm_\mu\in
  \real^{\mu\times\mu}$ the following holds:
  \begin{enumerate}
  \item All the eigenvalues of $\Lu_\nu$ are eigenvalues of $\Lm_{2\nu}$ and the
    corresponding eigenvectors, respectively $v\in\real^\nu$ and
    $w\in\real^{2\nu}$, are related by $w=
    \begin{bmatrix}
      \Pi v\\
      v
    \end{bmatrix}$;
  \item Eigenvalues and eigenvectors of $\Lu_\nu$ and $\Lm_{\mu}$ have the
    following closed form expression:
    \begin{equation}
      \begin{cases}
        \label{eq:eigN}
        \lambda_{\Lu_\nu} = 2-2\cos\left[(2k-1) \frac{\pi}{2\nu+1}\right],  \\[1ex]
        (v_k)_j =\sin\left[\dfrac{(\nu+j) (2k-1) \pi}{2 \nu + 1}\right], &j=1,\ldots,\nu ,\\
        & k=1,\ldots,\nu
      \end{cases}
    \end{equation}
    \begin{equation}
      \!\!\!\!\!\!\!\!\!\!\!\!\!\!\!\!\!\!
      \begin{cases}
        \label{eq:eigM}
        \lambda_{\Lm_\mu} = 2-2\cos\left(k\frac{\pi}{\mu+1}\right),\\
        (w_k)_j =\sin\left(\dfrac{j k \pi}{\mu + 1}\right), &j=1,\ldots,\mu,\\
        &k=1,\ldots,\mu.
      \end{cases}
    \end{equation}
  \end{enumerate} 
\end{proposition}

\begin{proof}
  To prove the first statement, consider a vector $w=
  \begin{bmatrix}
    w_1 \\
    w_2
  \end{bmatrix} =
  \begin{bmatrix}
    \Pi v\\
    v
  \end{bmatrix}$, with $v\in\real^\nu$.  We show that $w$ is an eigenvector of
  $\Lm_{2\nu}$ with eigenvalue $\lambda_N$ if and only if $v$ is an eigenvector
  of $\Lu_\nu$ with eigenvalue $\lambda_N$. Indeed, by exploiting the structure
  of $\Lm_{2\nu}$ we get
  \[
  \begin{split}
    \Lm_{2\nu}
    \begin{bmatrix}
      w_1 \\[1.2ex]
      w_2
    \end{bmatrix}
    &=
    \begin{bmatrix}
      \Pi \Lu_\nu \Pi+{e}_\nu{e}_\nu^T &  & -{e}_\nu{e}_1^T \\[1.2ex]
      -{e}_1{e}_\nu^T &  & \Lu_\nu+{e}_1{e}_1^T \\
    \end{bmatrix}
    \begin{bmatrix}
      \Pi v\\
      v
    \end{bmatrix}\\[1.2ex]
    &=
    \begin{bmatrix}
      \Pi \Lu_\nu \Pi \, \Pi v +{e}_\nu(\underbrace{{e}_\nu^T\Pi}_{{e}_1^T} \, v -  {e}_1^T \, v) \\[1.3ex]
      - {e}_1 (\underbrace{{e}_\nu^T\Pi}_{{e}_1^T} \, v)+ \Lu_\nu \, v + {e}_1
      (v)_1
    \end{bmatrix}\\[1.2ex]
    &=
    \begin{bmatrix}
      \Pi\Lu_\nu v +{e}_\nu [(v)_1 -  (v)_1 ]\\[1.3ex]
      \Lu_\nu v +{e}_1 [(v)_1 - (v)_1]
    \end{bmatrix}
    =
    \begin{bmatrix}
      \Pi\Lu_\nu v \\[1.3ex]
      \Lu_\nu v
    \end{bmatrix}
  \end{split}
  \]
  Now, notice that, if $\lambda_N$ is an eigenvalue of $\Lu_\nu$, i.e. $\Lu_\nu
  v=\lambda_N v$ for some nonzero $v\in\mathbb{R}^\nu$, then $w$ is an
  eigenvector of $\Lm_{2\nu}$ with the same eigenvalue. Viceversa, if $w =
  \begin{bmatrix}
    \Pi v\\
    v
  \end{bmatrix} \neq 0$ satisfies $\Lm_{2\nu} w=\lambda_N w$, then, from the
  last $\nu$ relations, it follows easily that $\Lu_\nu v=\lambda_N v$, thus
  concluding the first part of the proof.
	
  To prove the second statement, we observe that $M_\mu = 2 I - A_p$, where
  $A_p$ is the adjacency matrix of the path, reported in
  Appendix~\ref{subsec:adjac_path}. Thus, the eigenvalues of $M_\mu$ can be
  computed by summing $2$ to the eigenvalues of $A_p$ and the eigenvectors are
  the same, so that $\eqref{eq:eigM}$ follows.  We now show that the $\nu$
  eigenvalues and eigenvectors of $\Lu_\nu$ are the eigenvalues and eigenvectors
  of $\Lm_{2\nu}$ corresponding to the odd values of $k$ in \eqref{eq:eigM}. To
  see this, notice that the $\nu$ eigenvectors of $M_{2\nu}$ with structure $w=
  \begin{bmatrix}
    \Pi v \\
    v
  \end{bmatrix}$ satisfy $(w)_j=(w)_{2\nu-j+1}$, $j=1,\ldots,\nu$. Referring to
  the parametrization of $w$ in \eqref{eq:eigM}, we look for $k$s satisfying
  $\sin\left(\frac{j k \pi}{2 \nu + 1}\right)= \sin\left(\frac{(2\nu-j+1) k
      \pi}{2 \nu + 1}\right)$, $\forall j\in\until{\nu}$. This is easily
  verified for odd $k$s because the arguments of the sine functions on the two
  sides sum to $(2l+1)\pi$ for some integer $l$. This concludes the proof.
\end{proof}

Next, the eigenvalues of the Laplacian $L_n$ of a path graph of length $n$ are
expressed in closed form by relating them to the eigenvalues of the $\Lm_{n-1}$
matrix.
The following technical lemma gives the tools to directly compute the
eigenvalues of the laplacian matrix of a path graph.

\begin{lemma}
  The characteristic polynomials of $\Lu_\mu$ and $\Lm_\mu$, $\mu\in \natural$,
  satisfy the following relations,
 \[
  \begin{split}
    \det(sI-\Lu_\mu)&=(s-1)\det(sI-\Lm_{\mu-1})-\det(sI-\Lm_{\mu-2})\\
    \det(sI-\Lu_\mu)&=(s-2)\det(sI-\Lu_{\mu-1})-\det(sI-\Lu_{\mu-2})\\
    \det(sI-\Lm_\mu)&=(s-2)\det(sI-\Lm_{\mu-1})-\det(sI-\Lm_{\mu-2}).
  \end{split}
  \]
\end{lemma}
\normalsize
\begin{proof}
  The result is obtained by applying the Laplace expansion to compute the
  determinant of the matrices $(sI -\Lu_\mu)$ and $(sI-\Lm_\mu)$. In particular,
  the first and second relations are obtained by expanding respectively along
  the first and last row of $(sI -\Lu_\mu)$. The third relation is obtained by
  expanding along, e.g., the first row of $(sI -\Lm_\mu)$.
\end{proof}

\begin{proposition}[Eigenvalues of the Laplacian of a path]
  The characteristic polynomial of the Laplacian, $L_n$, of a path graph of
  length $n$ can be written as
  \[
  \det(sI-L_{n})= s \det(sI-\Lm_{n-1}).
  \]
  Thus, the eigenvalues of the Laplacian are given by
  \begin{equation*}
    \lambda_{L_n} = 2-2\cos\left((k-1) \frac{\pi}{n}\right), k=1,\ldots,n.
    \label{eigL}
  \end{equation*} 
\end{proposition}

\begin{proof}
  Applying the Laplace expansion for the computation of the determinant and
  using the results of the previous lemma, the following equalities hold:
  \[
  \begin{split}
    \det(sI-L_{n})
    &=(s-1)\det(sI-\Lu_{n-1})- \det(sI-\Lu_{n-2})\\
    &= (s-1)^2\det(sI-\Lm_{n-2})-2(s-1)\det(sI-\Lm_{n-3}) -\det(sI-\Lm_{n-4})\\
    &=(s^2-2s)\det(sI-\Lm_{n-2})- s\det(sI-\Lm_{n-3})\\     &=(s^2-2s)\det(sI-\Lm_{n-2})-s\left[(s-2)\det(sI-\Lm_{n-2})-\det(sI-\Lm_{n-1})\right]\\
    &= s\det(sI-\Lm_{n-1}).
  \end{split}
  \]
\end{proof}

\section{Reachability and observability of path and cycle graphs}
\label{sec:main_results}
In this section we completely characterize the reachability and observability of
path and cycle graphs.

\subsection{Reachability and observability of path graphs}
We characterize the reachability (observability) of a path graph by using the
PBH lemma in the form expressed in Corollary~\ref{cor:PBH_eigvec}. First, as
recalled in Remark~\ref{rmk:reach_observ_external_nodes}, a path graph is always
(reachable) observable from nodes $1$ or $n$. Next two lemmas give necessary and
sufficient conditions for reachability (observability) from a given subset of
nodes in terms of the $\Lu_\nu$ and $\Lm_\mu$ submatrices introduced in the
previous section.

\begin{lemma}
  A path graph of length $n$ is reachable (observable) from a node
  $i$\,$\in$\,$\fromto{2}{n-1}$ if and only if the matrices $\Lu_{i-1}$ and
  $\Lu_{n-i}$ do not have any common eigenvalue. The eigenvalues common to the
  two matrices are all and only the unreachable (unobservable) eigenvalues of
  the Laplacian $L_n$ from node $i$.
  \label{lem:share_eig_siso}
\end{lemma}

\begin{proof}
  Applying Corollary~\ref{cor:PBH_eigvec}, we have that $i$ is not reachable
  (observable) if and only if $L_{n} v = \lambda v$ and $e_i^T v = 0$ for some
  $\lambda$ and $v$. Equivalently, $i$ is not reachable (observable) if and only
  if there exists an eigenvector $v$ of $L_n$ with $v = [ v_1 ~0 ~v_2 ]^T$,
  $v_1\in\real^{i-1}$, $v_2\in\real^{n-1}$. Component-wise this is written as
  \begin{equation}
    \begin{split}
      \Lu_{i-1} v_1 &= \lambda v_1\\
      (v_1)_{i-1} + (v_2)_{1} &= 0\\
      \big(\Pi\Lu_{n-i}\Pi\big) \, v_2 &= \lambda v_2.
    \end{split}
    \label{eq:PBH_path_single}
  \end{equation}
  The necessary condition follows easily by the above equations. Indeed, if
  $\Lu_{i-1}$ and $\Lu_{n-i}$ have at least one common eigenvalue $\lambda_0$
  with corresponding eigenvectors respectively $v_{1 0}$ and $v_{2 0}$, then the
  conditions in the above equations are satisfied for $v = [v_{1 0} ~0 ~\rho
  v_{2 0}]$ and $\lambda = \lambda_0$, where $\rho \in \real$ is just a scaling
  factor to satisfy $(v_{1 0})_{i-1} + \rho(v_{2 0})_{1} = 0$.

  To prove the converse we proceed by analyzing when these three conditions are
  satisfied. The first equation in \eqref{eq:PBH_path_single} is verified in two
  cases: i) $v_1=0$ and $\lambda$ arbitrary, and ii) $v_1$ and $\lambda$
  respectively eigenvector and eigenvalue of $\Lu_{i-1}$. From the first
  condition it follows easily that $(v_2)_{1}=0$ and, using
  Lemma~\ref{lem:nonzero_comp}, that $v_2 = 0$. Therefore, the only possibility
  to have unreachability (unobservability) is ii). Now, using the second
  equation it follows easily that $v_2 \neq 0$, and, using the third equation,
  that $\Pi v_2$ must be an eigenvector of $\Lu_{n-i}$ corresponding to the same
  eigenvalue $\lambda$ of $\Lu_{i-1}$. This concludes the first part of the
  proof.

  The fact that the unreachable (unobservable) eigenvalues of $L_n$ from node
  $i$ are all and only the eigenvalues common to $\Lu_{i-1}$ and $\Lu_{n-i}$
  follows straight from the previous argument. Indeed, by definition the
  unreachable (unobservable) eigenvalues are all and only the ones satisfying
  the condition in Corollary~\ref{cor:PBH_eigvec} and, thus, the equations in
  \eqref{eq:PBH_path_single}.
\end{proof}

\begin{remark}[Paths with odd number of nodes]
  A straightforward consequence of the previous lemma is that a path graph with
  an odd number, $n$, of nodes is not reachable (observable) from the central
  node. Also, $(n-1)/2$ eigenvalues are unreachable (unobservable) from that
  node, namely the eigenvalues of $L_n$ that are also eigenvalues of
  $\Lu_{(n-1)/2}$. The corresponding $(n-1)/2$ unreachable (unobservable)
  eigenvectors of $L_n$ are of the form $[v^T \; 0 \; \text{-}v^T \Pi]^T$ with
  $v$ being the eigenvectors of $\Lu_{(n-1)/2}$ as in \eqref{eq:eigN}. \oprocend
\end{remark}

A generalization to the multi input (output) case is given in the following
lemma.

\begin{lemma}
  \label{lem:share_eig_mimo}
  A path graph of length $n$ is reachable (observable) from the set of control
  (observation) nodes $I_o = \{i_1, \ldots, i_m\}$ if and only if the matrices
  $\Lu_{i_1-1}$, $\Lm_{i_2 - i_1 - 1}$, $\ldots$, $\Lm_{i_m - i_{m-1} - 1}$ and
  $\Lu_{n-i_m}$ do not have common eigenvalues. The eigenvalues common to the
  matrices are all and only the (unreachable) unobservable eigenvalues of $L_n$
  from the set $I_o$. 
\end{lemma}

\begin{proof}
  We proceed as in the single input (output) case and apply the PBH Lemma in the
  version expressed in Corollary~\ref{cor:PBH_eigvec}.
  The path is not reachable (observable) from the set $I_o= \{i_1, \ldots,
  i_m\}$ if and only if there exists an eigenvector $v$ of $L_n$ with
  $(v)_{i_1}=\ldots=(v)_{i_m}=0$, so that $v = [v_{1}^T ~0 ~v_{2}^T \ldots ~0
  ~v_{m}^T]^T$ for suitable vectors $v_{1}\in\real^{i_1-1}, \ldots,
  v_{m}\in\real^{i_m-1}$. This is equivalent to
  \begin{equation}
    \begin{split}
      N_{i_1-1}v_{1} &= \lambda v_{1}\\
      (v_{1})_{i_1-1} +  (v_{2})_1  &= 0\\
      M_{i_2-i_1-1} v_{2} &= \lambda v_{2}\\
      (v_{2})_{i_2-1} +  (v_{3})_1  &= 0\\
      &\vdots  \\
      (v_{m-1})_{i_m-1} + (v_{m})_{1} &= 0\\
      \big(\Pi\Lu_{n-i_m}\Pi\big) v_{m}&= \lambda v_{m} \,.
    \end{split}
    \label{eq:eigvec_mimo}
  \end{equation}
  The proof follows by using the same arguments in
  Lemma~\ref{lem:share_eig_siso}.
\end{proof}

We are now ready to completely characterize the reachability (observability) of
a path by means of simple algebraic rules from number theory. For the sake of
clarity, we state the theorem for path graphs of length $n$, where $n$ has a
prime factorization with distinct odd prime factors. The general case follows
straight and is discussed in a remark.

\begin{theorem}[Path reachability and observability]
  \label{thm:main_thm_path}
  Given a path graph of length $n$, let $n = 2^{n_0} \prod_{\nu =1}^k p_\nu$ be
  a prime number factorization for some $k\in \natural$ and distinct (odd) prime
  numbers $p_1, \ldots, p_{k}$. The following statements hold:
  \begin{enumerate}
  \item the path is not completely reachable (observable) from a node
    $i\in\fromto{2}{n-1}$ if and only if
    \[
    (n-i)\; \eqmod{p} \;(i-1)
    \]
    for some odd prime $p$ dividing $n$;
  \item the path is not completely reachable (observable) from a set of nodes
    $I_o = \{i_1, \ldots, i_m\} \subset \fromto{2}{n-1}$ if and only if
    \[
    \begin{split}
      2 (i_1-1) + 1 \eqmod{p} (i_2-i_1) \eqmod{p} 
      \ldots \eqmod{p} i_{m}-i_{m-1} & \eqmod{p} 2 (n-i_m) + 1,
    \end{split}
    \]
    for some odd prime $p$ dividing $n$;
  \item \label{pt:set_unobs_p} for each odd prime factor $p\in\{p_1,\ldots,
    p_k\}$ of $n$,
    the path is not reachable (observable) from each set of nodes $I_o^p = \{
    \ell p -
    \frac{p-1}{2}\}_{\ell\in\until{\frac{n}{p}}}$
    with the following unreachable (unobservable) eigenvalues
    \begin{equation}
      \begin{split}
        \lambda_{\nu} = 2-2\cos\left((2 \nu-1) \frac{\pi}{{p}}\right), \qquad
        \nu \in\until{\frac{p-1}{2}};
      \end{split}
      \label{eq:unobs_eigs_thm}
    \end{equation}
    and unreachable (unobservable) eigenvectors
    \begin{equation}
      \begin{split}
        V_{\nu} =
        \begin{bmatrix}
          v_{\nu}^T & 0 &-(\Pi v_{\nu})^T &-v_{\nu}^T & 0
          & \ldots &(-1)^{\frac{n}{p}}(\Pi v_{\nu})^T
        \end{bmatrix}^T,
      \end{split}
      \label{eq:unobs_eigvecs_thm}
    \end{equation}
    where $v_{\nu}\in\real^{(p-1)/2}$ is the eigenvector of $\Lu_{(p-1)/2}$
    corresponding to the eigenvalue $\lambda_{\nu}$ for $\nu
    \in\until{(p-1)/2}$; and
  \item if node $i$ belongs to $I_o^{q_j}=\{ \ell q_j -
    \frac{q_j-1}{2}\}_{\ell\in\until{\frac{n}{q_j}}}$ for $l\leq k$ distinct
    prime factors $q_{1}\neq \ldots \neq q_{l}$ of $n$, then the set of
    unreachable (unobservable) eigenvalues from node $i$ is given by
    \begin{equation*}
      \begin{split}
        \lambda_{\nu} = 2-2\cos\left((2 \nu-1) \frac{\pi}{{q_1\cdot \ldots \cdot
              q_l}}\right), \qquad \nu \in\until{\frac{(q_1\cdot \ldots \cdot
            q_l)-1}{2}}.
      \end{split}
    \end{equation*}
    Also, the orthogonal complement to the reachable subspace, $(X_r)^\perp$,
    (respectively the unobservable subspace, $X_{no}$) is spanned by all the
    corresponding eigenvectors of the form
    \begin{equation*}
      \begin{split}
        V_{\nu} =
        \begin{bmatrix}
          v_{\nu}^T & 0 &-(\Pi v_{\nu})^T &-v_{\nu}^T & 0
          & \ldots &(-1)^{\frac{n}{p}}(\Pi v_{\nu})^T
        \end{bmatrix}^T,
      \end{split}
    \end{equation*}
    where $v_{\nu}\in\real^{((q_1\cdot\ldots\cdot q_l)-1)/2}$ is the eigenvector
    of $\Lu_{((q_1\cdot\ldots\cdot q_l)-1)/2}$ corresponding to the eigenvalue
    $\lambda_{\nu}$ for $\nu \in\until{((q_1\cdot\ldots\cdot q_l)-1)/2}$.
  \end{enumerate} 
\end{theorem}

\begin{proof}
  Using Lemma~\ref{lem:share_eig_siso} we have that the path graph is not
  completely reachable (observable) from node $i$ if and only if $\Lu_{i-1}$ and
  $\Lu_{n-i}$ have at least one common eigenvalue. Therefore, using
  Proposition~\ref{prop:eigs_Lu_Lm}, we have that it must hold $2 - 2 \cos
  \frac{(2 j_1 - 1) \pi}{2 (i-1) + 1} = 2 - 2 \cos \frac{(2 j_2 - 1) \pi}{2
    (n-i) + 1}$, for some $j_1\in\until{i-1}$ and $j_2\in\until{(n-i)}$. In the
  admissible range of $j_1$ and $j_2$ the cosine arguments are less than $\pi$
  so that the cosine is one to one. Thus, the equality holds if and only if
  $\frac{(2 j_1 - 1)}{2 (i-1) + 1} = \frac{(2 j_2 - 1)}{2 (n-i) + 1}$. Two
  integers $j_1$ and $j_2$ satisfying this equation exist in the admissible
  range if only if $2 (i-1) + 1$ and $2 (n-i) + 1$ are not coprime, that is, if
  and only if $GCD(2 i-1, 2n-2i+1)$ is greater than one. Now, $GCD(2 i-1,
  2n-2i+1)$ is odd because $2 i-1$ and $2n-2i+1$ are. Therefore, we can write $2
  i-1 = p \alpha_1$ and $2n-2i+1 = p \alpha_2$ with $p$, $\alpha_1$ and
  $\alpha_2$ odd.
  This is equivalent to $n + n-2i+1 = p \alpha_2$ and, since $p$ divides $n$,
  $p$ must divide also $(n-2i+1)$ thus concluding the first part of the proof.

  To prove statement (ii), we have by Lemma~\ref{lem:share_eig_mimo} that the
  path graph is not reachable (observable) from the set $I_o = \{i_1, \ldots,
  i_m\}$ if and only if the matrices $\Lu_{i_1-1}$, $\Lm_{i_2 - i_1 - 1}$,
  $\ldots$, $\Lm_{i_m - i_{m-1} - 1}$ and $\Lu_{n-i_m}$ do not have common
  eigenvalues. The proof follows by using again
  Proposition~\ref{prop:eigs_Lu_Lm} and the same arguments as in the single node
  case.

  To prove statement (iii), we start observing that the set of nodes $I_o^p = \{
  \ell p - \frac{p-1}{2}\}_{\ell\in\until{\frac{n}{p}}}$, is the set of all
  nodes satisfying condition in (i) for a given $p\in\{p_1, \ldots,
  p_k\}$. Using Lemma~\ref{lem:share_eig_mimo}, we have that the unreachable
  (unobservable) eigenvalues from this set of nodes are the common eigenvalues
  to $\Lu_{(p-1)/2}$ and $\Lm_{p-1}$. From Proposition~\ref{prop:eigs_Lu_Lm}, it
  follows easily that the common eigenvalues between $\Lu_{(p-1)/2}$ and
  $\Lm_{p}$ are all the eigenvalues of $\Lu_{(p-1)/2}$ and have the form in
  equation~\eqref{eq:unobs_eigs_thm}. As regards the unreachable (unobservable)
  eigenvectors, using equation~\eqref{eq:eigvec_mimo} for this special set
  $I_o$, it follows straights that the eigenvectors have zero components as in
  equation~\eqref{eq:unobs_eigvecs_thm}. To prove that the nonzero components
  have that special structure in equation~\eqref{eq:unobs_eigvecs_thm}, we
  observe that they must be eigenvectors of respectively $\Lu_{(p-1)/2}$,
  $\Lm_{p-1}$, $\ldots$,$\Lm_{p-1}$ and $\Pi\Lu_{(p-1)/2}\Pi$. From point (i) of
  Proposition~\ref{prop:eigs_Lu_Lm}, the eigenvectors of $\Lu_{(p-1)/2}$,
  $\Lm_{p-1}$, $\ldots$,$\Lm_{p-1}$ and $\Pi\Lu_{(p-1)/2}\Pi$ are respectively
  $\alpha_0 v_\nu$, $\alpha_1[(\Pi v_\nu)^T \;v_\nu^T]^T$, $\ldots$,
  $\alpha_{\frac{n}{p}-1}[(\Pi v_\nu)^T \;v_\nu^T]^T$ and $\alpha_{\frac{n}{p}}
  \Pi v_\nu$, where $v_\nu$ is an eigenvector of $\Lu_{(p-1)/2}$ and
  $\alpha_\mu\in\real$, $\mu\in\{0, \ldots, \frac{n}{p}\}$. Finally, using again
  equation~\eqref{eq:eigvec_mimo}, $\alpha_\mu = -\alpha_{\mu-1}$,
  $\mu\in\until{\frac{n}{p}}$, so that the proof follows by choosing $\alpha_0 =
  1$.

  The proof of statement (iv) follows from the definition of unreachable
  (unobservable) eigenvalues and eigenvectors, and arguments as in the previous
  statements.
\end{proof}

\begin{remark}[General version of Theorem~\ref{thm:main_thm_path}]
  \label{rmk:general_main_thm_path}
  In the general case of a path graph of length $n = 2^{n_0} \prod_{\nu =1}^k
  p_\nu$, where $p_1, \ldots, p_k$ are not all distinct, statement (i) and (ii)
  of Theorem~\ref{thm:main_thm_path} continue to hold in the same form. As
  regards statement (iii), it still holds in the same form, but it can also be
  strengthen with a slight modification. That is, for each multiple factor
  $\bar{p}$ with multiplicity $\bar{k}$, the statement continues to hold if
  $\bar{p}$ is replaced by $\bar{p}^\alpha$ with
  $\alpha\in\until{\bar{k}}$. 
  Statement (iv) holds if for each prime factor $\bar{p}$ with multiplicity
  $\bar{k}$ we check if node $i$ belongs not only to $I_o^{\bar{p}}$, but also
  to each $I_o^{\bar{p}^\alpha}$ with $\alpha\in\until{\bar{k}}$. Consistently
  the unreachable (unobservable) eigenvalues and eigenvectors considered in the
  statement must be constructed by using $\bar{p}^{\bar{\alpha}}$ instead of
  $\bar{p}$, where $\bar{\alpha} = \max_{\alpha} \{\alpha \in \until{\bar{k}} |
  i\in I_o^{\bar{p}^\alpha}\}$. \oprocend
\end{remark}

The following corollary follows straight from Theorem~\ref{thm:main_thm_path}
and characterizes all and only the path graphs that are observable from any
node.

\begin{corollary}[Reachable (observable) paths from any node]
  Given a path graph of length $n=2^k$ for some $k\in\natural$, then the path is
  reachable (observable) from any node.\oprocend
\end{corollary}

Next, we provide a simple routine giving a graphical interpretation of the
results of the theorem. We describe the routine for paths with simple
factorization leaving the generalization to a specific example.
We proceed by associating a unique symbol to each set of nodes defined at point
\ref{pt:set_unobs_p} of Theorem ~\ref{thm:main_thm_path} for a given $p$.  Thus,
each group of nodes sharing the same unreachable (unobservable) eigenvalues has
the same symbol and nodes of different groups (and so associated to different
unreachable (unobservable) eigenvalues) have different symbols. Formally, let $n
= 2^{n_0} \prod_{\nu =1}^k p_\nu$ for some $n_0\in\natural$ and $p_1, \ldots,
p_k$ prime integers. At the beginning of the procedure we initialize all the
nodes without any symbol. For any $p_\nu$, $\nu\in\until{k}$, we partition the
nodes into $n / p_\nu$ groups of $p_\nu$ nodes and assign the same symbol to all
the nodes in position $i = j p_\nu - \frac{p_\nu-1}{2}$,
$j\in\until{\frac{n}{p_\nu}}$.
A set of nodes from which the path is reachable (observable) is obtained by
selecting any node without symbols, if there are any, or a set of nodes having
no symbols in common.

Three examples for $n=6$ (even), $n=15$ (odd) and $n=9$ (multiple factor) are
shown respectively in Figure~\ref{fig:path6}, in Figure~\ref{fig:path15} and in
Figure~\ref{fig:path9}.
In Figure~\ref{fig:path6} nodes with the triangle symbol are unable to
reconstruct the state of the network by themselves. Indeed, they share the same
unreachable (unobservable) eigenvalue $\lambda=1$. In view of the previous
results, focusing on node $i_1=2$, notice that $\Lu_{i-1}=N_1=[1]$ (whose
eigenvalue is $1$), $\Lu_{n-i}=\Lu_{4}$ and its eigenvalues are: $\{0.12, 1,
2.35, 3.53\}$. The common eigenvalue is of course $\lambda=1$. The
unreachability (unobservability) can be checked more easily using the test
$(n-i)=4\; \eqmod{3} \;1=(i-1)$.

\begin{figure}[h!]
  \centering
  \includegraphics[height=.08\linewidth]{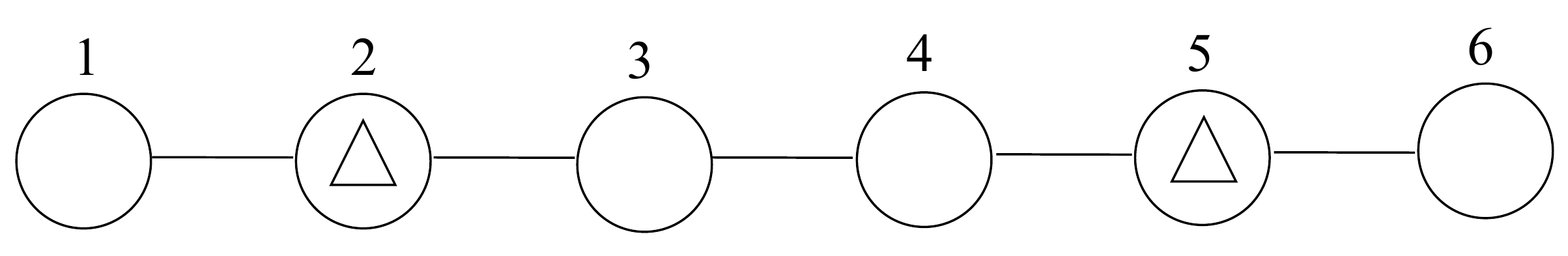}%
  \caption{Reachable (observable) nodes for a path with $n=6$ nodes.}
  \label{fig:path6}
\end{figure}

In Figure~\ref{fig:path15} nodes with the triangle belong to the set
$I_o^{p_1}$, with $p_1=3$, and nodes with the square to $I_o^{p_2}$, with
$p_2=5$. The ``triangle nodes'' share the same (unreachable) unobservable
eigenvalue $\lambda=1$ ($\Lu_{i-1}=N_1=[1]$), while two unreachable
(unobservable) eigenvalues, $0.3820$ and $2.6180$, are associated to the
``square nodes''. Finally, the central node has both the triangle and square
symbols.

\begin{figure}[h!]
  \centering
  \includegraphics[height=.2\linewidth]{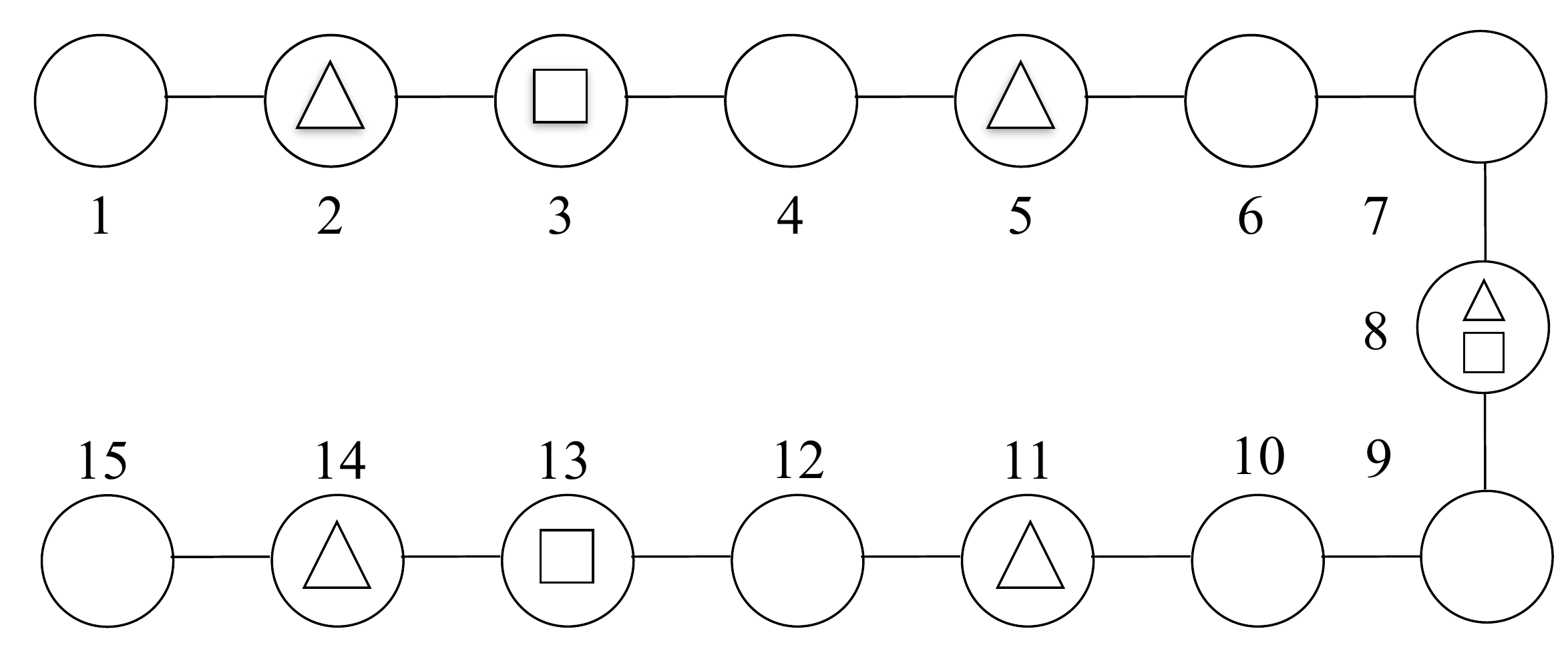}%
  \caption{Reachable (observable) nodes for a path with $n=15$ nodes.}
  \label{fig:path15}
\end{figure}

In Figure~\ref{fig:path9} we consider a path of length $n=9=3^2$, with $p=3$
being a multiple factor. Nodes with the triangle belong to the set $I_o^{p_1}$,
with $p_1=3$. To the central node is associated both a triangle (since it
belongs to $I_o^3$) and a square since it is the unique node in $I_o^{p_1^2}$,
with $p_1^2=9$. The ``triangle nodes'' share the same (unreachable) unobservable
eigenvalue $\lambda=1$ ($\Lu_{i-1}=N_1=[1]$), while the central node has four
unreachable (unobservable) eigenvalues being the eigenvalues of $\Lu_4$. This
example suggests how to associate symbols in the general case when $n$ has
multiple factors. That is, for each multiple factor $\bar{p}$ with multiplicity
$\bar{k}$, we use $\bar{k}$ different symbols, one for each ${\bar{p}}^\alpha$,
$\alpha\in\until{\bar{k}}$.

\begin{figure}[h!]
  \centering
  \includegraphics[height=.11\linewidth]{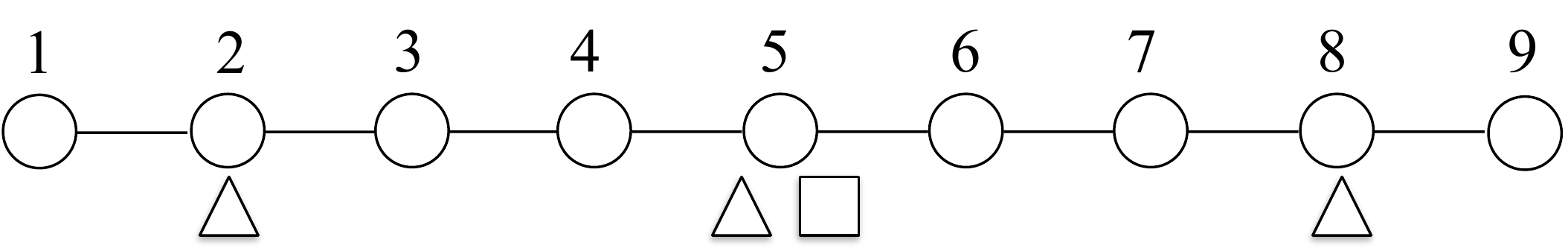}%
  \caption{Reachable (observable) nodes for a path with $n=9$ nodes.}
  \label{fig:path9}
\end{figure}

\subsection{Reachability and observability of cycle graphs}
Next, we characterize the reachability (observability) of a cycle graph. We
start with a negative result, namely that a cycle graph is not reachable
(observable) from a single node. First, we need a well known result in linear
systems theory \cite{PJA-ANM:94}.

\begin{lemma}
  If a state matrix $A\in\real^{n\times n}$, $n\in\natural$, has an eigenvalue
  with geometric multiplicity $\mu\geq 2$, then for any $B\in\real^n$
  (respectively $C\in\real^{1\times n}$) the pair $(A,B)$ is unreachable
  (respectively the pair $(A,C)$ is unobservable).\oprocend
\end{lemma}

As shown in Appendix~\ref{subsec:circul_mat}, all but at most two eigenvalues of
the Laplacian of the cycle have geometric multiplicity two. Thus, applying the
previous lemma next proposition follows.

\begin{proposition}
  A cycle graph is not completely reachable (observable) from a single node for
  any choice of the control (observation) node. Furthermore,
  $\lceil\frac{n-1}{2}\rceil$ eigenvalues are unreachable
  (unobservable).  
\end{proposition}

\begin{proof}
  The unreachability (unobservability) follows directly by noting that all but
  at most two eigenvalues of the cycle have geometric multiplicity two and by
  applying the previous lemma. By using the PBH lemma, it follows straight that
  the unreachable (unobservable) eigenvalues are all but the zero eigenvalue.
\end{proof}

Recall that, without loss of generality, we label the nodes of the cycle so that
the undirected edges are $(i, (i\,\text{mod}\,(n) + 1))$ for
$i\in\until{n}$. Following the same line as in the reachability (observability)
analysis of path graphs, we can prove the following lemma.

\begin{lemma}
  \label{lmm:share_eig_cycle}
  A cycle graph of length $n$ is reachable (observable) from the set of nodes
  $I_o = \{i_1, \ldots, i_m\}$ if and only if the matrices $\Lm_{i_2 - i_1 -
    1}$, $\ldots$, $\Lm_{i_m - i_{m-1} - 1}$ and
  $\Lm_{(i_1-i_m-1)\textrm{mod}n}$ do not have common eigenvalues. 
\end{lemma}

\begin{proof}
  Without loss of generality, we set node $i_1=1$, so that the $\Lm$ matrices in
  the statement of the theorem become $\Lm_{i_2 - i_1 - 1}$, $\ldots$, $\Lm_{i_m
    - i_{m-1} - 1}$ and $\Lm_{n-i_m}$. Following the same lines of the proof of
  Lemma~\ref{lem:share_eig_mimo}, loss of reachability (observability) is
  equivalent to the existence of a nonzero solution $v = [0 \;v_{1}^T \;0
  \;v_{2}^T \ldots \;0 \;v_{m}^T]^T$ to:
  \[
  \begin{split}
    (v_{1})_{1} +  (v_{m})_{n-i_m}  &= 0\\
    M_{i_2-i_1-1} v_{1} &= \lambda v_{1}\\
    &\vdots  \\
    (v_{m-1})_{i_m-1} + (v_{m} )_{1} &= 0\\
    M_{n-i_m} v_{m}&= \lambda v_{m}.
  \end{split}
  \]
  Now, if $v_j = 0$ for some $j\in\until{m}$, using Lemma~\ref{lem:nonzero_comp}
  the only solution of the above system is $v=0$. With this condition in hand,
  it follows straight that $v$ is a (nonzero) solution if and only if all the
  $v_j$ are nonzero eigevectors of $M_{i_2 - i_1-1}, \ldots, M_{n-i_m}$ with
  common eigenvalue $\lambda$, thus concluding the proof.
\end{proof}

It is worth noting that, due to the symmetry of the cycle, the reachability
(observability) properties are determined by the relative distance between each
pair of consecutive control (observation) nodes. The following theorem parallels
Theorem~\ref{thm:main_thm_path}. As for the path, we state the theorem for cycle
graphs of length $n$, where $n$ has a prime factorization with distinct prime
factors. The general case follows straight from similar arguments as in
Remark~\ref{rmk:general_main_thm_path}.

\begin{theorem}[Cycle reachability and observability]
  Given a cycle graph of length $n$, let $n = \prod_{\nu =1}^k p_\nu$ be a prime
  number factorization for some $k\in \natural$ and distinct prime numbers $p_1,
  \ldots, p_{k}$ (including the integer $2$). The following statements hold:
  \begin{enumerate}
  \item \label{pt:cond_unobs_cycle} the cycle graph is reachable (observable)
    from the set of nodes $I_o = \{i_1, \ldots, i_m\}$ if and only if
    \begin{equation}
      GCD\big( (i_2 - i_1), (i_3 - i_2), \ldots, (n+i_1-i_m)\big) = 1;
      \label{eq:GCD_cycle}
    \end{equation}
  \item for each prime factor $p$ of $n$ and for each fixed
    $\kappa\in\until{p}$, the set of nodes $I_o^p = \{\kappa + \ell p\}_{\ell\in
      \{0, \ldots, \frac{n}{p} - 1\}}$, is unreachable (unobservable) with the
    following unreachable (unobservable) eigenvalues
    \begin{equation}
      \lambda_{\nu} = 2-2\cos\left(\nu \frac{\pi}{p}\right), \qquad \nu\in\until{p-1}
      \label{eq:unobs_eigs_thm_cycle}
    \end{equation}
    and, for $\kappa=1$,
    eigenvectors 
    \begin{equation}
      \begin{split}
        V_{\nu} = \begin{bmatrix} 0 &w_{\nu}^T &\ldots &0
          &w_{\nu}^T\end{bmatrix}^T,
      \end{split}
      \label{eq:unobs_eigvecs_thm_cycle}
    \end{equation}
    where $w_{\nu}\in\real^{(p-1)}$ is the eigenvector of $\Lm_{(p-1)}$
    corresponding to the eigenvalue $\lambda_{\nu}$.
  \item if a set of control (observation) nodes $I_o$ with cardinality greater
    than $1$, satisfies
    \[
    GCD\big( (i_2 - i_1), (i_3 - i_2), \ldots, (n+i_1-i_m)\big) = q_1 \cdot
    \ldots \cdot q_l
    \]
    where $q_{1}\neq \ldots \neq q_{l}$ are $l\leq k$ distinct prime factors of
    $n$, then the set of unreachable (unobservable) eigenvalues from $I_o$ is
    given by
    \begin{equation*}
      \lambda_{\nu} = 2-2\cos\left(\nu \frac{\pi}{q_1\cdot\ldots\cdot q_l}\right), \qquad \nu\in\until{(q_1\cdot\ldots\cdot q_l)-1}.
    \end{equation*}
    Also, without loss of generality, setting $i_1=1$, the orthogonal complement
    to the reachable subspace, $(X_r)^\perp$, (respectively the unobservable
    subspace, $X_{no}$) is spanned by all the corresponding eigenvectors
    \begin{equation*}
      \begin{split}
        V_{\nu} = \begin{bmatrix} 0 &w_{\nu}^T &\ldots &0
          &w_{\nu}^T\end{bmatrix}^T,
      \end{split}
    \end{equation*}
    where $w_{\nu}\in\real^{(q_1\cdot\ldots\cdot q_l)-1}$ is the eigenvector of
    $\Lm_{(q_1\cdot\ldots\cdot q_l)-1}$ corresponding to the eigenvalue
    $\lambda_{\nu}$.  
  \end{enumerate}
\end{theorem}

\begin{proof}
  We provide just a sketch of the proof since it follows the same line as the
  proof of Theorem~\ref{thm:main_thm_path}.
  Statement~(i) is proven by using Lemma~\ref{lmm:share_eig_cycle} and the
  structure of the eigenvalues of the $M$ matrices given in
  Proposition~\ref{prop:eigs_Lu_Lm} with the same argument as in
  Theorem~\ref{thm:main_thm_path} (i) and (ii).
  To prove statement~(ii), we start observing that the set of nodes $I_o^p = \{
  \kappa + \ell p\}_{\ell\in\until{\frac{n}{p}}}$, is the set of all nodes
  satisfying condition in (i) for a given $p\in\{p_1, \ldots, p_k\}$. Using
  Lemma~\ref{lmm:share_eig_cycle}, we have that the unreachable (unobservable)
  eigenvalues from this set of nodes are the eigenvalues of $\Lm_{p-1}$. The
  proof follows by the same arguments as in Theorem~\ref{thm:main_thm_path}
  (iii).
  Finally, statement~(iii) follows straight.
\end{proof}

Next corollaries provide respectively an easy way to choose two control
(observation) nodes to get reachability (observability) for any cycle length and
the class of cycle graphs (lengths) for which reachability (observability) is
guaranteed for any pair of nodes.
\begin{corollary}\label{cor:cycle1}
  Any cycle graph is reachable (observable) from two adjacent nodes.\oprocend
\end{corollary}

\begin{corollary}
  A cycle graph of length $n$ is reachable (observable) from any pair of nodes
  if and only if $n$ is prime. \oprocend
\end{corollary}

As for the path, we provide a simple routine giving a graphical interpretation
of the results of the theorem. Again, we give the procedure for the case of
simple factors.

We will mark each unreachable (unobservable) node with a different symbol. Let
$n = \prod_{\nu =1}^k p_\nu$ for some $k\in\natural$ and $p_1, \ldots, p_k$
distinct prime integers (here we include $2$ among the $p_\nu$ as well). At the
beginning of the procedure all the nodes are initialized without any symbol. For
any $p_\nu$, $\nu\in\until{k}$ and $i\in \until{{p_\nu}}$, partition the nodes
into $n / p_\nu$ groups of $p_\nu$ nodes and assign the same symbol to all the
nodes in position $i + k\cdot p_\nu$, $j\in\until{\frac{n}{p_\nu}-1}$. Nodes
with the same symbol have the same unreachable (unobservable) eigenvectors, in
the sense that controlling (observing from) all the nodes at the same time gives
the same unreachable (unobservable) eigenvectors. A set of nodes from which the
cycle is reachable (observable) is obtained by selecting any subset of nodes
having no symbols in common.

In Figure~\ref{fig:cycle15} there are two symbols for each node. This is because
$n=15=3\cdot 5$. Symbols closer to the nodes have periodicity $5$ and the others
have periodicity $3$. Notice the ease of design using the above procedure. For
example $\{4, 13\}$ and $\{8, 14\}$ are unreachable (unobservable) pairs since
they share respectively the square and the parallelogram, while $\{2, 13\}$ and
$\{5, 12\}$ are reachable (observable) pairs. Finally, notice that two
neighboring nodes have always different symbols in accordance with the result in
Corollary~\ref{cor:cycle1}.

\begin{figure}[htbp]
  \centering
  \includegraphics[height=.35\linewidth]{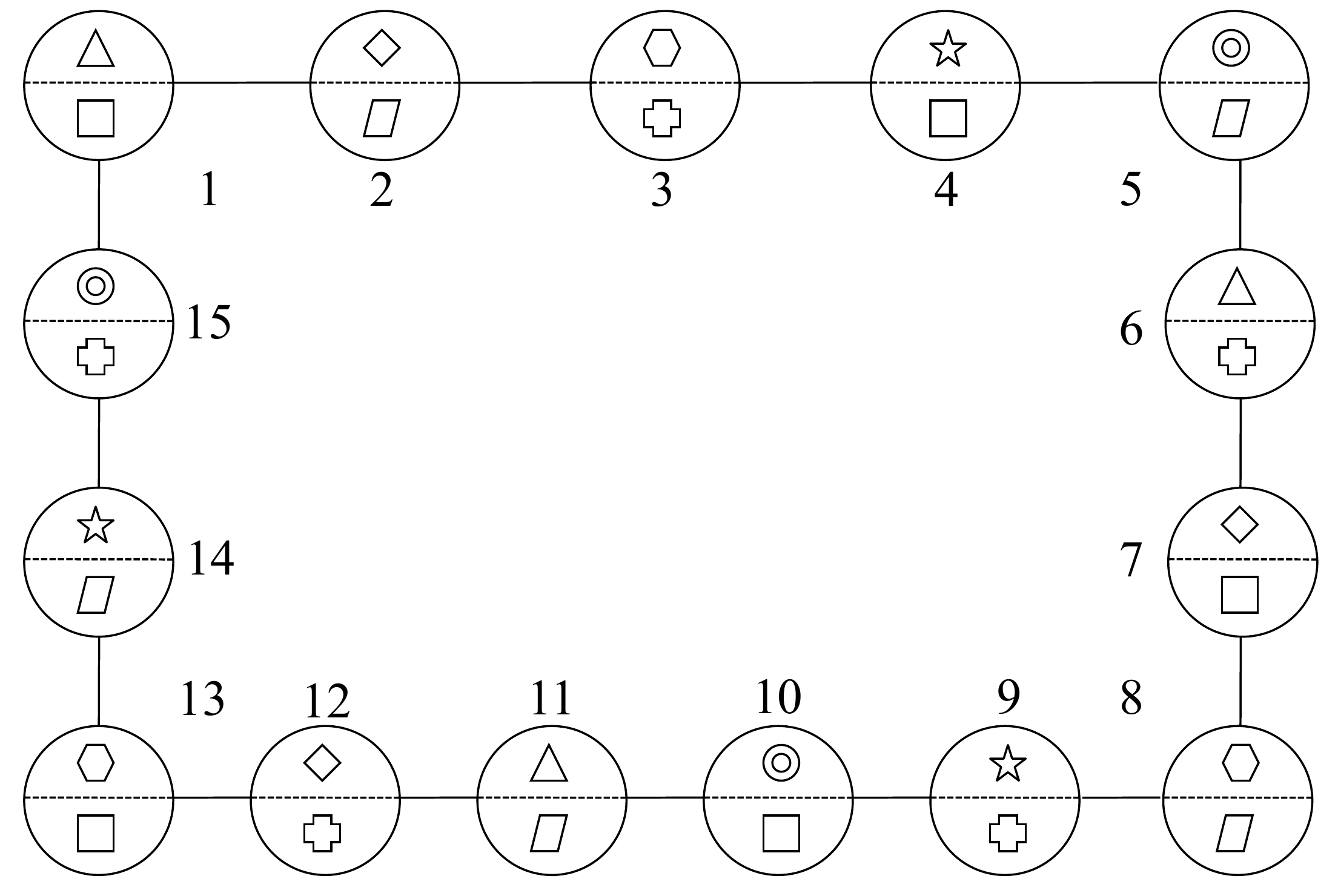}%
  \caption{Graphical interpretation of the observability of a cycle with $15$
    nodes.}
  \label{fig:cycle15}
\end{figure}

\section{Conclusions}
\label{sec:conclusions}
In this paper we have characterized the reachability (observability) of path and
cycle graphs in terms of simple algebraic rules from number theory. In
particular, we have shown what are all and only the unreachable (unobservable)
set of nodes and provided simple routines to select a set of control
(observation) nodes that guarantee reachability (observability).

Promising avenues for further research include the extension of the proposed
methodologies to more complex graphs having paths and cycles as constitutive
graphs (e.g., grid, torus and cylinder graphs).

\appendix

\subsection{Adjacency matrix of path graphs}
\label{subsec:adjac_path}
Eigenvalues and eigenvectors of the adjacency matrix, $A_p \in\real^{n\times
  n}$, of a path graph of length $n$ can be easily computed \cite{FRG:98}. Here
we briefly summarize some of the steps.

Consider the matrix
\[
A_p=
\begin{bmatrix}
  0       & 1    &0 & \dots     & 0  \\
  1       & 0     & 1  &       &  \vdots  \\
  \vdots  &  \ddots   & \ddots  & \ddots& 1  \\ \\
  0   & \dots & & 1   & 0 \\
\end{bmatrix}
\]

By definition, an eigenvalue $\lambda$ and a corresponding eigenvector $v \in
\real^n$ of $A_p$ satisfy $(v)_{i-1}-\lambda (v)_i +(v)_{i+1}=0$ with
$(v)_0=(v)_{n+1}=0$ and at least one $(v)_i, i\in\until{n}$, nonzero.
Set $(v)_1=a$ and build the sequence $(v)_i$ according to
$(v)_{i+1}=-(v)_{i-1}+\lambda (v)_i$ (e.g. $(v)_2=\lambda a$, $(v)_3=\lambda^2 a
- a $, $(v)_4=\lambda^3 a-2\lambda a$, \ldots, $(v)_\ell=p_\ell(\lambda)
a=(\lambda\cdot p_{\ell-1}(\lambda)-p_{\ell-2}(\lambda))a$).
Imposing $p_{n+1}(\lambda)=0$ one finds $\lambda_k=2\cos\left(k
  \frac{\pi}{n+1}\right)$ as all possible values giving $(v)_{n+1}=0$ with
nonzero $a$.
Now, choose $a_k=\sin\left(k \frac{\pi}{n+1}\right)$, the corresponding
eigenvector can be expressed componentwise $({v_k})_i=\sin \left( i\cdot k
  \frac{\pi}{n+1} \right)$ according the recursive formula above and simple
trigonometric rules.

\subsection{Circulant matrices and eigenstructure of the Laplacian of a cycle
  graph}
\label{subsec:circul_mat}

An $n\times n$ matrix $\ C$ of the form
\[
C=
\begin{bmatrix}
  c_0     & c_{n-1} & \dots  & c_{2} & c_{1}  \\
  c_{1} & c_0    & c_{n-1} &         & c_{2}  \\
  \vdots  & c_{1}& c_0    & \ddots  & \vdots   \\
  c_{n-2}  &        & \ddots & \ddots  & c_{n-1}   \\
  c_{n-1}  & c_{n-2} & \dots  & c_{1} & c_0 \\
\end{bmatrix}
\]
is called a \emph{circulant matrix} \cite{PJD:94}.
A circulant matrix is fully specified by the first column $c=[c_0, \ldots,
c_{n-1}]^T$ of $\ C$. The other columns are obtained by a cyclic permutation of
the first.
The eigenvalues of a circulant matrix can be expressed in terms of the
coefficients $c_0, \ldots, c_{n-1}$ \cite{PJD:94}: $\lambda_j = \sum_{k=0}^{n-1}
\omega^{jk} c_k, \ \omega=e^{i\frac{2\pi}{n}}$, where here $i$ represents the
imaginary unit.

The Laplacian matrix of a cycle graph is a special case of this family
corresponding to $c_0=2$, $c_1=c_{n-1}=-1$, $c_j=0$, $j=2,\ldots,n-2$:
\begin{equation*}
  \label{eigRing}
  \lambda_j =w^{j0}2 - e^{i\frac{2\pi}{n}j(n-1)} - e^{i\frac{2\pi}{n}j}=2-2\cos \left(\frac{2\pi}{n}j\right), \
\end{equation*}
$j=0,1, \ldots, n-1$. Notice that the eigenvalues $\lambda_0=0$ and
$\lambda_{\frac{n}{2}}=4$ (only if $n$ is even) are simple, all the others
verify $ \lambda_j = \lambda_{n-j}$.

\end{document}